\input ./style/arxiv-general.cfg
\documentclass[aap,MSNbibl,dvips]{arximspdf}
\makeatletter
   \@ifpackageloaded{graphicx}{}{\usepackage{graphicx}}
\makeatother
\usepackage{mathbh}

\doi{10.1214/14-AAP1069}
\volume{25}
\issue{6}
\pubyear{2015}
\firstpage{3139}
\lastpage{3161}
\docsubty{FLA}

\makeatletter
\newcommand{\eqref}[1]{(\ref{#1})}
\newcommand{\ds}{\displaystyle}
\newtheorem{cor}[thm]{Corollary}
\newtheorem{lem}[thm]{Lemma}
\newproclaim{ques}[thm]{Question}
\newproclaim{dfn}[thm]{Definition}
\newproclaim{rem}[thm]{Remark}
\makeatother

\begin{document}
\begin{frontmatter}

\title{Extinction window of mean field branching annihilating random
walk\thanksref{T3}}
\runtitle{Extinction window of mean field BARW}

\begin{aug}
\author[A]{\fnms{Idan}~\snm{Perl}\thanksref{T1}\ead[label=e1]{perli@post.bgu.ac.il}},
\author[B]{\fnms{Arnab}~\snm{Sen}\thanksref{T2}\ead[label=e2]{arnab@math.umn.edu}}
\and
\author[A]{\fnms{Ariel}~\snm{Yadin}\corref{}\thanksref{T1}\ead[label=e3]{yadina@math.bgu.ac.il}\corref{}}
\runauthor{I. Perl, A. Sen and A. Yadin}
\affiliation{Ben-Gurion  University of the Negev,  University of Minnesota\\ and Ben-Gurion  University of the Negev}
\address[A]{I.~Perl\\
A.~Yadin\\
Department of Mathematics\\
Ben-Gurion University of the Negev \\
Beer Sheva 8410501\\
Israel \\
\printead{e1}\\
\phantom{E-mail: }\printead*{e3}}
\address[B]{A.~Sen \\
Department of Mathematics\\
University of Minnesota \\
Minneapolis, Minnesota  55455\\
USA\\
\printead{e2}}
\end{aug}
\thankstext{T3}{Support for this work came from
EPSRC Grant EP/GO55068/1.}
\thankstext{T1}{Supported in part  by
Grant  2010357 from the United States--Israel Binational Science
Foundation (BSF).}
\thankstext{T2}{Supported in part by NSF
Grant   DMS-14-06247.}

%
\received{\smonth{10} \syear{2013}}
%
\revised{\smonth{8} \syear{2014}}

\begin{abstract}
We study a model of growing population that competes for resources.
At each time step, all existing particles reproduce and the offspring
randomly move to neighboring sites.
Then at any site with more than one offspring, the particles are annihilated.
This is a nonmonotone model, which makes the analysis more difficult.

We consider the \textit{extinction window} of this model in the finite
mean-field case, where there are $n$
sites but movement is allowed to any site (the complete graph).
We show that although the system survives for exponential time, the
extinction window is logarithmic.
\end{abstract}

\end{frontmatter}

\section{Introduction}

\subsection{The model}

Perhaps the most classical population model is the\break \textit{Galton--Watson
branching process}.
Originally devised to model the survival of aristocratic patrilineal
surnames, the Galton--Watson
process may be described as follows:
start with one existing particle.
At every time step, all existing particles reproduce an independent
number of offspring and die out.
The main question is then, what is the probability that the system
survives forever?
By use of generating functions it is fairly simple to analyze this
model, and
in fact it is well known that in a Galton--Watson process with
offspring distribution
$L$, the probability of extinction is given by the unique minimal
solution of the equation
$s = \mathbb{E}[ s^L ]$ in the interval $(0,1]$.
Moreover, the solution $q$ satisfies $q=1$ if and only if $\mathbb
{E}[L] \leq1$; see, for example, \cite{AN72,LyonsPeres} for a
thorough treatment.

To make matters more interesting, one might add some geometry,
by having the particles not only branch (reproduce) but also
move in some underlying graph.
This is the \textit{branching random walk} model, which is described as follows:
start with one particle at some origin vertex $o$ in graph $G$.
At each time step, all existing particles reproduce an independent number
of offspring and die out. All offspring now independently choose a
random neighbor of their
parent's vertex, and move to that new position.
Thus a specific lineage of particles performs a random walk on $G$.
A different way to view this model is as a tree-indexed random walk
(see \cite{BP94A,BP94B} for more on tree-indexed random walks) where the
domain tree is the tree of lineage formed by a Galton--Watson process.
See the pioneering
work of Biggins \cite{B77} and the survey by Shi \cite{S08}.

Both models mentioned above exhibit some sort of monotonicity, enabling
coupling arguments. For example,
put in an imprecise way, if one has more particles, the branching
random walk is more likely
to be recurrent. The additional particles only help it return to the origin.

Let us now introduce the model we work with, which we dub \textit{branching-annihilating
random walk}, or $\mathsf{BARW}$ for short.
Start with a single particle at some origin vertex $o$ of a graph $G$.
At each time step, all particles independently reproduce (or branch)
into a random number of offspring. These offspring then each choose
independently a random
neighbor of their parent's vertex and move to that neighbor. (So far,
everything is
identical to the branching random walk.)
Finally, at every vertex at which there is more than one particle,
these particles are eliminated
(this is the annihilation phase).

$\mathsf{BARW}$ is a model for population reproduction in some
geometry, with a competition for resources.
The annihilation phase can be viewed as there being only enough
resources for one particle
at every vertex of the underlying graph.

Let us stress that the difficulty in analyzing $\mathsf{BARW}$ stems mainly from
the lack of monotonicity.
Adding particles may on the one hand assist in the ultimate survival of
the system, but may
also hinder the survival, as these additional particles may compete for
resources and annihilate
others, resulting in too few particles to survive.


It is most convenient to work with Poisson distributed offspring, so
for simplicity we
will restrict to this distribution.

\begin{dfn}
\label{dfn:BARW}
Let $\lambda> 1$ be a real number.
Let $G$ be a graph, and let $o \in G$ be some vertex.
We define \textit{branching-annihilating random walk} on~$G$,
starting at~$o$, with parameter $\lambda$,
or $\mathsf{BARW}_{G,o}(\lambda)$, as the following Markov process on
subsets of~$G$.

Let $(L_{t,j})_{t,j=1}^{\infty}$ be i.i.d. Poisson-$\lambda$ random
variables.
Start with $B_0 =  \{o \}$.
For every $t\geq0$, given $B_t \neq\varnothing$, define $B_{t+1}$ as follows.

Suppose that $B_t =  \{x_1, \ldots, x_m  \}$.
For every $1 \leq j \leq m$, let $y_{j,1}, \ldots, y_{j,L_{t,j}}$ be
independent
vertices chosen uniformly from the set $ \{y \dvtx y \sim x_j
 \}$ (the neighbors of~$x_j$ in $G$).
Define $Z_{t+1} \dvtx  G \to\mathbb{R}$ by $Z_{t+1}(x) = \sum_{j=1}^m
\sum_{i=1}^{L_{t,j} } \mathbh{1}_{ \{y_{j,i} = x  \} }$.
This is the number of offspring that have moved to $x$.

Finally, let $B_{t+1} =  \{x \dvtx Z_{t+1}(x) = 1  \}$.
In the case that $B_t = \varnothing$, then $B_{t+1} = \varnothing$ as well.
\end{dfn}

\subsection{Main questions and results}

As stated above, $\mathsf{BARW}$ lacks monotonicity, and thus it is
not easy to analyze.
However, it seems reasonable to ask the following immediate questions
regarding the long-term behavior.
Some of these questions are being studied by the authors in a separate work,
for the case of $G$ being the infinite $d$-regular tree.

Suppose $G$ is an infinite transitive graph. If $\lambda$ is either
too big or too small,
one may dominate $\mathsf{BARW}$ by a sub-critical Galton--Watson
process. Thus
we are guaranteed extinction in either case.
(This is not surprising, as too little offspring do not give a good
enough chance of survival,
and too many offspring create too much annihilation, thus again ruining
the chance of survival.)

The immediate questions that arise regard a super-critical interval of survival:
\begin{itemize}
\item Do there exist $\lambda_c^- \leq\lambda_c^+$ such that
for $\lambda\in(\lambda_c^-, \lambda_c^+)$ there is positive
probability of survival forever,
and for $\lambda\notin[\lambda_c^-, \lambda_c^+]$ there is
extinction a.s.?

\item If such an interval exists, what happens at the critical values
$\lambda= \lambda_c^-$
and $\lambda= \lambda_c^+$?

\item Can $\lambda_c^-, \lambda_c^+$ be identified?
\end{itemize}

In this paper we consider $\mathsf{BARW}$ in the finite graph setting,
and specifically on the complete graph.
Of course, there is always a positive probability of extinction in one step
on a finite graph, so on a finite graph $\mathsf{BARW}$ will a.s. die
out at some finite time.
However, we may consider $\mathsf{BARW}$ on a sequence of finite
graphs with size tending to infinity,
and try to understand asymptotic properties of the process for large graphs.

In this work we consider the mean-field case, where the sequence under
consideration
is the complete graph on $n$ vertices as $n \to\infty$.

Our first result states that $\mathsf{BARW}$ on the complete graph has
an exponentially large expected lifetime.

\begin{thm}
\label{thm:extinction time}
For every $\lambda>1$ there exists $c = c(\lambda)>0$
such that the following holds for all $n \in\mathbb{N}$.
Consider $\mathsf{BARW}$ on the complete graph on $n$ vertices,
and let $X_t = |B_t|$ be the number of particles at time $t$. Let
\[
T_0 = \inf \{t \ge0\dvtx  X_t = 0 \}.
\]
Then, for each $0< x < n$,
\[
\mathbb{E}[ T_0 | X_0 = x] \geq c e^{c n}.
\]
\end{thm}

Our main result regards the ``window'' of extinction.
It is not difficult to see that for $\mathsf{BARW}$ on the complete
graph on $n$ vertices,
the number of particles will oscillate for a long time around the value
$\mathrm{eq}: = \frac{\log\lambda}{\lambda} n$. We call it the
\textit{quasi-stable} state, which is obtained by solving for the state
$x$ such that $\mathbb{E}[ X_1 | X_0 =x] =x$. Below it the chain has
an upward drift whereas there is a downward drift if the chain goes
above the quasi-stable state.
Our next result considers how long it takes the process to go extinct,
once it has
been conditioned to do so; that is, how many steps did it take the
process to reach $0$
particles, at the last excursion it made below the equilibrium point
$\frac{\log\lambda}{\lambda} n$?

\begin{thm}
\label{thm:main thm}
For every $\lambda>1$ and $0 < \varepsilon< \frac{\log\lambda
}{\lambda}$, there exists $C = C(\lambda, \varepsilon) > 0$ such
that the following holds for all $n \in\mathbb{N}$.

Consider $\mathsf{BARW}$ on the complete graph on $n$ vertices, and
let $X_t = |B_t|$ be the number of particles at time $t$. Let
\[
T_0 = \inf \{t \ge0\dvtx  X_t = 0 \} \quad \mbox{and}\quad
T^+_{\mathrm{eq}- \varepsilon n} = \inf \biggl\{t \ge0 \dvtx  X_t \ge
\frac{\log\lambda}{\lambda} n - \varepsilon n \biggr\}.
\]
Then for each $0 \leq x < \frac{\log\lambda}{\lambda} n -
\varepsilon n$,
\[
C^{-1} \log(1+x) \le\mathbb{E}\bigl[ T_0|X_0 = x,  T_0 <T^+_{\mathrm{eq}-\varepsilon n} \bigr] \leq C
\log(1+x).
\]
\end{thm}

\begin{rem}
Though the above theorem holds for any $\lambda>1$, the conditioned
chain $(X_t)_{ t \ge0} | T_0 < T^+_{\mathrm{eq}- \varepsilon n}$
exhibits remarkably different behaviors in two distinct regimes of the
parameter $\lambda$: (i) $\lambda$ is close to $1$, and (ii) $\lambda$ is large; see
Figures \ref{fig:1}~and~\ref{fig:2}. Our proof is
general enough to tackle both regimes simultaneously.
\end{rem}

\begin{figure}

\includegraphics{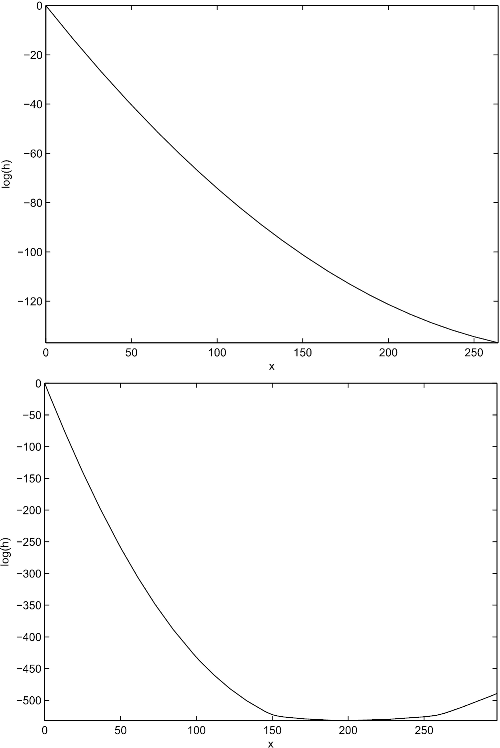}

\caption{Plot of $\log h(x)$ vs $x$ where $h(x)=\mathbb{P}_x[T_0<T_{\mathrm{eq}- \varepsilon n}^+] $ for $n=1200$,
$\varepsilon= 0.05$ and $\lambda= 1.5$ (left) and $\lambda= 6$
(right). Note that for $\lambda= 1.5$, $h$ is monotonically
decreasing, but $\log h$ is not linear, so $h$ can not be expressed as
$C \exp(-c x)$.
On the other hand, for $\lambda= 6$, the function $h$
is not even monotone---it first decreases, and then it increases near
$\mathrm{eq}- \varepsilon n$.}\label{fig:1}
\vspace*{-6pt}
\end{figure}

\begin{figure}

\includegraphics{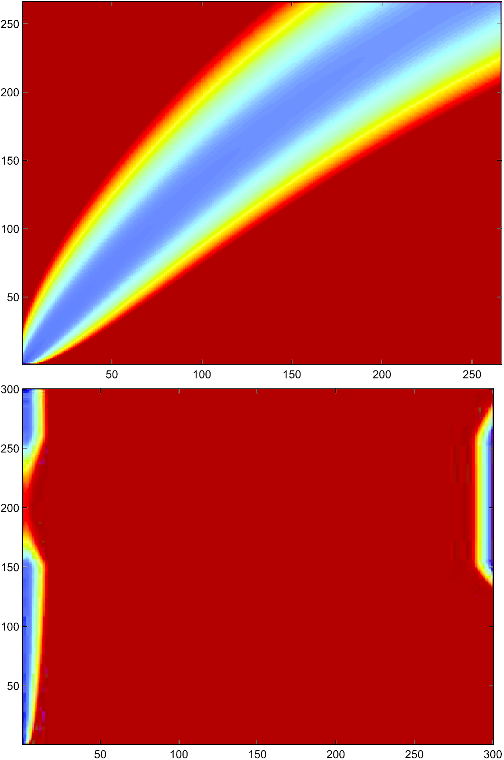}

\vspace*{-3pt}
\caption{Transition probabilities of the \emph{tilted} chain $P(\cdot,\cdot|T_0<T_{\mathrm{eq}-\varepsilon n}^+ )$ for $n=1200$,
$\varepsilon= 0.05$ and $\lambda= 1.5$ (left) and $\lambda= 6$
(right). The probabilities are represented by colors---the blue
represents high values, and the red represents small values. For
$\lambda=1.5$, from any $x$ in the tilted chain, the walker goes down
by a multiplicative factor with high probability. But for $\lambda=6$,
the transition matrix is highly concentrated. For some $x$, the tilted
chain goes up with high probability. For some $x$, it goes down with
high probability. For a few $x$'s, the transition distribution is
bimodal!}\label{fig:2}
\vspace*{-6pt}
\end{figure}

It would be interesting to find out whether, for a fixed $n$, the
expected extinction time of the conditioned chain $ \mathbb{E}[ T_0| X_0 = x, T_0 <T^+_{\mathrm{eq}- \varepsilon n}]$
is decreasing
with respect to $\lambda$.

\subsection{Similar models and further questions}

$\mathsf{BARW}$, or rather a continuous time versions, have been
studied before;
see, for example, \cite{BG85,BWD91,Sudbury90}.
However, most focus on survival of the process, or stationary measures.

On the other hand, recently there has been considerable interest among
the physicists to study the behavior of a finite population evolving
under some stochastic dynamics near its extinction time and
particularly to find ``most probable or optimal path to extinction''
\cite{FBSF11,SFBS11}.

To best of our knowledge this is the first work to study the
``extinction window'' for $\mathsf{BARW}$;
that is, the length of the last path to extinction.
As our results show, at least in the mean-field case, this window is
\textit{much} smaller than
the lifetime of the system, indicating that extinction is a
``catastrophic'' phenomenon, meaning that it occurs abruptly
in a very short time frame.

Our analysis makes heavy use of the fact that on the complete graph,
the geometry plays no role,
so that $\mathsf{BARW}$ can actually be seen as a Markov chain on
$ \{0,1,\ldots,n \}$, making the model simpler.
It would be very interesting to understand the expected lifetime and
extinction window in
other finite graph settings. More specifically:

\begin{ques}
Let $(G_n, o_n)_n$ be a sequence of finite rooted graphs converging in
the local weak topology \cite{AS04} to a limiting rooted graph $(G, o)$.
Consider $\mathsf{BARW}$ on $G_n$ with Poisson-$\lambda$ offspring:
\begin{itemize}
\item Is it true that there exist critical $\lambda_c^- \leq\lambda
_c^+$ such that
if $\lambda\in(\lambda_c^-, \lambda_c^+)$,
then the expected lifetime is exponentially large in $|G_n|$,
and if $\lambda\notin[\lambda_c^-, \lambda_c^+]$,
the expected lifetime is much smaller (perhaps logarithmic)?

\item For which $\lambda$ does $\mathsf{BARW}$ on $G_n$ have a
logarithmically small extinction window?
That is, for which $\lambda$ does there exist small enough $\eta>0$
so that
conditioned on extinction before reaching above $\eta|G_n|$ particles,
the conditioned
process has logarithmically small expected lifetime?
\end{itemize}
\end{ques}

The above question is open even for a sequence of finite $d$-regular
graphs with increasing girths (whose local limit is the infinite
$d$-regular tree).

\subsection{Comparison with SIS model and variants}

It has been suggested that the $\mathsf{BARW}$ is similar in spirit to
the SIS infection model.
In the SIS model, all vertices in a graph are either infected or not.
The infected vertices infect their neighbors at a certain rate, and
every vertex recovers from
infection at a different rate, these rates being parameters of the model.
The discrete time version of this model may have two interpretations:
we may allow only one particle to act at every time step, which is the
discrete time backbone of the continuous time chain,
or allow all particles to act at the same time.
A similar variant may have been used in the $\mathsf{BARW}$ model.

It turns out that the SIS and $\mathsf{BARW}$ models are sensitive to
these kind of local modifications, and
we do not see a way to relate them.
In the one-particle-at-a-time versions on complete graphs, the chains
are birth and death chains, meaning that they
are Markov chains on $ \{0,1,\ldots,n \}$ with transition
probabilities restricting movement
only between states at distance $1$.
This makes the analysis simpler using the available tools for such
chains; see, for example, \cite{AldousFill}, \cite{Feller}, Chapter~XVII.5,
\cite{LPW09}, Chapter~2.4.
Let us give a brief account of this analysis.

\subsubsection{$\mathsf{BARW}$ one particle at a time}
In the continuous time $\mathsf{BARW}$ model, particles die at rate
$1$ and give off a particle to a uniform
vertex at rate $\lambda> 1$. When two particles are at a vertex, they
instantly annihilate one another.
Consider the number of living individuals and the discrete backbone of
this continuous chain as a discrete
time Markov process. Note that at each time step either one particle
dies or a new one is added, or nothing is changed.
If there are $x$ living individuals,
with probability $\frac{\lambda}{1+\lambda} \cdot (1 - \frac
{x}{n}  )$, a particle is added to an empty vertex, and the number
of individuals increases by $1$;
with probability $\frac{1}{1+\lambda}$, a living individual dies and
the number of total individuals decreases by $1$;
with remaining probability $\frac{\lambda}{1+\lambda} \cdot\frac{x}{n}$,
a~particle is added to an occupied vertex, resulting in annihilation,
so the number of total living individuals decreases by $1$.

To sum up, the transition probabilities of this chain are given by
\[
P_B(x,y) =\cases{\ds \dfrac{1}{1+\lambda} +
\dfrac{\lambda}{1+\lambda} \cdot\dfrac
{x}{n}, & \quad $x>0, y=x-1$, \vspace*{4pt}
\cr
\ds
\dfrac{\lambda}{1+\lambda} \cdot \biggl(1- \dfrac{x}{n} \biggr), & \quad $x>0, y=x+1$,
\vspace*{4pt}\cr
1, & \quad $y=x=0$.}
\]

\subsubsection{SIS one particle at a time}
In the SIS model the difference is that annihilation is replaced by
coalescence. Analogously to the above,
infected individuals recover with rate $1$ and infect a neighbor at
rate $\lambda>1$.
So considering the discrete backbone of the total number of infected vertices,
with probability $\frac{1}{1+\lambda}$, a vertex recovers and the
total number decreases by $1$;
with probability\vspace*{1pt} $\frac{\lambda}{1+\lambda} \cdot\frac{x}{n}$, an
infected vertex is infected, resulting in no change
to the total number of infected vertices; with probability $\frac
{\lambda}{1+\lambda} \cdot (1 - \frac{x}{n})$,
a\vspace*{1pt} healthy vertex is infected, and the total number increases by $1$.
The following is a summary of the transition probabilities for the SIS model:
\[
P_S(x,y) = %
\cases{\ds\frac{1}{1+\lambda}, & \quad $x>0, y=x-1$,
\vspace*{4pt}
\cr
\ds\frac{\lambda}{1+\lambda} \cdot\frac{x}{n}, &\quad  $x>0, y=x$,
\vspace*{4pt}
\cr
\ds\frac{\lambda}{1+\lambda} \cdot \biggl(1- \frac{x}{n} \biggr), &
\quad $x>0, y=x+1$,\vspace*{4pt}
\cr
1, &\quad $y=x=0$.}
\]

One now sees that there is an additional drift downward
for the $\mathsf{BARW}$ model that is not present in the SIS model.

\subsubsection{Extinction window}
In this subsection, whenever we talk
about the $\mathsf{BARW}$ and the SIS model, we refer to their
one-particle-at-a-time version.
For $\mathsf{BARW}$ and the SIS model, the quasi-stable states are
given by $\mathrm{eq}_B = \frac{\lambda- 1}{2 \lambda} n$ and
$\mathrm{eq}_S = \frac{\lambda- 1}{ \lambda} n$, respectively.
Clearly, for both these chains, the expected extinction time is at
least exponential in $n$, that is, $\mathbb{E}_x[ T_0] \ge c e^{c n}$
for any $x$, since we can find $\delta>0$ such that between the states
$0$ and $\delta n$ each of the chains can be coupled from below with a
simple random walk with bias away from zero. Thanks to the standard
results on birth and death chains regarding hitting probabilities
(\cite{AldousFill}, \cite{Feller}, Chapter XVII.5, \cite{LPW09}, Chapter~2.4),
the extinction window is also easy to calculate for these
models. Let us first talk about the SIS model. The transition
probabilities of the chain conditioned on the event $\{ T_0 <
T_{\mathrm{eq}_S - \varepsilon n}\}$ can be obtained via Doob's $h$-transform,
\[
\hat{P}_S( x, y) = P_S(x, y) \frac{ \mathbb{P}_y[T_0 < T_{\mathrm{eq}_S - \varepsilon n} ] }{\mathbb{P}_x[T_0 < T_{\mathrm{eq}_S -
\varepsilon n}]}.
\]
Let $M = \mathrm{eq}_S - \varepsilon n$. We will use the following
standard notation for the jump probabilities of a birth and death
chain: $p_x = P_S(x, x+1)$, $q_x = P_S(x, x-1)$ and $r_x = P_S(x,x)$.
Then $\mathbb{P}_x[ T_0 < T_M] = \frac{\varphi(M) - \varphi(x)}{
\varphi(M)}$ where $\varphi(x) =   \break \sum_{m=0}^{x-1} \prod_{j=1}^m
\theta_j $ for $x>1$ and $\varphi(0)=0$ and $\theta_x = q_x/p_x$.
Note that the tilted chain $\hat P_S$ is again a birth and death chain
on $0,1, 2, \ldots, M$ with jump probabilities $\hat p_x = \hat P_S(
x, x+1)$, $ \hat q_x = \hat P_S( x, x-1)$ and $\hat r_x = \hat P_S(x,x)
= r_x$.

We have
%
\begin{eqnarray}
\frac{\hat p_x}{\hat q_x} &=& \frac{p_x}{q_x} \cdot\frac{\sum_{m=x+1}^{M} \prod_{j=1}^m \theta_j}{\sum_{m=x-1}^{M} \prod_{j=1}^m
\theta_j}
\nonumber
\\[-8pt]
 \label{ratio1}
 \\[-8pt]
\nonumber
&=&  \frac{\theta_{x+1} + \theta_{x+1} \theta_{x+2}+ \cdots+ \theta
_{x+1} \theta_{x+2} \cdots\theta_M }{1+ \theta_{x} + \theta_{x}
\theta_{x+1}+ \cdots+ \theta_{x} \theta_{x+1} \cdots\theta_M}
\end{eqnarray}
for $0< x< M$, $ \frac{1}{\lambda} < \theta_x < \frac{1}{1+ \lambda
\varepsilon}$. Writing $z = \min( x+ C_1 \log n, M)$ for sufficiently
large $C_1$, we can approximate the ratio in \eqref{ratio1} by
%
\begin{equation}
\label{ratio2}
\frac{\theta_{x+1} + \theta_{x+1} \theta_{x+2}+ \cdots+ \theta
_{x+1} \theta_{x+2} \cdots\theta_z }{1+ \theta_{x} + \theta_{x}
\theta_{x+1}+ \cdots+ \theta_{x} \theta_{x+1} \cdots\theta_{z}} + O\bigl(n^{-1}\bigr).
\end{equation}
Using the fact that $|\theta_x - \theta_y| \le C_2 \frac{ |x -
y|}{n}$ for all $x, y < M$, we can write \eqref{ratio2} as
\[
\frac{\theta_{x} + \theta_{x}^2+ \cdots+ \theta_{x}^{z-x} }{1+
\theta_{x} + \theta_{x}^2+ \cdots+ \theta_{x}^{z-x +1} } + o(1) = \theta_x + o(1),
\]
where the error term $o(1)$ is uniform in $0< x < M$.

Hence, for sufficiently large $n$, the tilted chain $\hat P_S$ can be
coupled from above by a lazy simple random walk with holding
probability $\frac{1}{1+\lambda}$ and with a bias toward~$0$.
Therefore, we conclude that there exists a constant $C>0$ such that
\[
x \le \mathbb{E}_x[ T_0 | T_0 <
T_{\mathrm{eq}_S - \varepsilon n} ] \le C x\qquad  \mbox{for each } 0< x < \mathrm{eq}_S
- \varepsilon n.
\]
We can prove a similar result on the extinction window for the $\mathsf{BARW}$ one-particle-at-a-time model following exactly the same
arguments as above.

\subsubsection{SIS all particles at once}

As mentioned, in this note we consider $\mathsf{BARW}$ with all
particles reproducing at once.
The analogous SIS version could be defined as follows.
At every time step, every infected vertex infects $\operatorname{Poi}(\lambda)$
uniformly chosen neighbors (perhaps some chosen more than once).
Vertices not re-infected then recover.
This is the same as replacing annihilation in $\mathsf{BARW}$ with coalescence.
So the SIS model is the same as a branching-coalescing random walk.

When considered on the complete graph,
if there are $x$ infected vertices, every vertex receives $\operatorname{Poi}(\frac{\lambda x}{n})$ infections,
so is left infected at the next time step with probability $(1-
e^{-\lambda x/n})$, independently for all vertices.
Thus, given that there are $x$ infected vertices at time $t$, the
number of infected vertices at time $t+1$
has $\operatorname{Bin}(n, (1-e^{-\lambda x/n}) )$ distribution.

Note that the equation $n(1-e^{-\lambda x/n}) = x$ has exactly two
solutions in $[0,n]$, one which is
at $x=0$, and the other being the equilibrium of this model. Since
$n(1-e^{-\lambda x/n}) - x$ is maximized
at $x = \mathrm{eq}= n \frac{\log\lambda}{\lambda}$ and since
this maximum is positive,
we have that the equilibrium of SIS is larger than $\mathrm{eq}$, the
equilibrium of $\mathsf{BARW}$.

Analysis of the SIS model's extinction window is another possible
future direction of research.

\subsection{Preliminaries and notation} \label{subs:notation}

It will be much simpler to use the following equivalent form of
$\mathsf{BARW}$ on the complete
graph on $n$ vertices.
(Here is where the mean-field structure makes the analysis much simpler.)
Given that $|B_t| = x$, that is, there are $x$ particles at time $t$,
every particle branches into Poisson-$\lambda$ particles, and each of these
chooses a new vertex, independently, and uniformly among all $n$ vertices.
Thus, due to the summability of the Poisson distribution,
at the branching phase every vertex receives an independent
Poisson-$\frac{\lambda x}{n}$ number of particles.
In the annihilation phase only those vertices with exactly one particle
survive to the next step,
which happens at a given vertex with probability $b(x) : = \frac
{\lambda x}{n} e^{- \lambda x / n}$.

Thus, we have just shown that if $(X_t)_t$ is the number of existing
particles in $\mathsf{BARW}$
on the complete graph on $n$ vertices, then $(X_t)_t$ is a
Markov chain with transitions given by
\[
\mathbb{P}[ X_{t+1} = y | X_t = x ] = \mathbb{P}
\bigl[ \operatorname{Bin}\bigl(n, b(x) \bigr) = y \bigr] = \pmatrix{n \cr y}
b(x)^y \bigl(1-b(x)\bigr)^{n-y}.
\]
This observation will be central in what follows.

We use the notation $\mathbb{P}_x$ and $\mathbb{E}_x$ to denote the
probability measure
and expectation of $\mathsf{BARW}$ on the complete graph on $n$ vertices
with (Poissonian) offspring mean $\lambda$ and with $X_0=x$.

Let $\lambda>1$. Consider the Galton--Watson process with offspring
distribution
$L\sim\operatorname{Poi}(\lambda)$.
It is well known
that there exists a number $q = q(\lambda) \in(0,1)$
such that the process dies out with probability $q$,
and that $q$ is the unique fixed point of the equation $s=\mathbb
{E}[s^L] = e^{-\lambda(1-s)}$ in $(0,1)$.
Also, since $q$ is the probability of extinction, it is clear that
$q(\lambda)$ is a continuous strictly decreasing function of $\lambda$;
see, for example, \cite{LyonsPeres}.

Throughout, we make extensive use of the following inequalities, which
are easy to verify:
\begin{itemize}
\item For any\vspace*{1pt} $t\in(0,1)$ and $n\in\mathbb{N}$, $e^{-nt}\geq(1-t)^n$.
\item For any $0\leq t\leq\frac{1}{2}$, we have $\sqrt{1-2t}\geq1-t-2t^2$.
\item For any $0\leq t\leq\frac{1}{2}$, we have $1-t\geq
e^{-1+\sqrt{1-2t}}$.
\item The last two inequalities can be combined to deduce $1-t\geq
e^{-t(1+2t)}$.
\end{itemize}

$x \wedge y$ denotes the minimum of $x,y$, and $x \vee y$ denotes the
maximum of $x,y$.

We also make use of the stopping times
\[
T_x^+ = \inf \{t \ge0 \dvtx X_t \geq x \}.
\]

Another tool we will use is the following standard large deviations result
concerning binomial random variables.
For $0<\xi<1$,
\[
\mathbb{P}\bigl[ \operatorname{Bin}(n,b) < \xi nb \bigr] \leq\exp \biggl(- nb
\cdot \frac{(1-\xi)^2}{4} \biggr).
\]

\section{The extinction time for unconditional chain}

In this section we prove Theorem~\ref{thm:extinction time}.

Let $\tau^+_{\varepsilon n}$ be the \textit{return} time to one of the
sites in $[\varepsilon n,n]$,
\[
\tau^+_{\varepsilon n}:=\inf \{t\geq1\dvtx  X_t \geq\varepsilon n \}.
\]

For the proof of Theorem~\ref{thm:extinction time} we do not require
the full strength
of the following lemma, but it will also be required in the sequel.
Recall from Section~\ref{subs:notation} that given $\lambda>1$, $q
= q(\lambda) \in(0,1)$ is the Poisson dual parameter, that is, the
unique number satisfying $\lambda e^{- \lambda} = qe^{-q}$.

\begin{lem}
\label{lem:sub and super}
Let $0<\varepsilon<\frac{1}{2\lambda}$ and small enough such that
$\lambda e^{-\lambda\varepsilon}>1$. Let $\lambda_1=\lambda
e^{-\lambda\varepsilon},\lambda_2=\lambda(1+2\lambda\varepsilon
)$, and define $q_1=q(\lambda_1), q_2=q(\lambda_2)$. Then
\[
\frac{q_2^x-q_2^{\varepsilon n}}{1-q_2^{\varepsilon n}}\leq g(x) \leq \frac{q_1^x-q_1^n}{1-q_1^n},\qquad  0\leq x<\varepsilon n,
\]
where $g(x):=\mathbb{P}_x[T_0<T_{\varepsilon n}^+]$.
\end{lem}

\begin{pf}
Denote $b(x)= \frac{\lambda x}{n}e^{-{\lambda x}/{n}}$.
On $X_t=x$, we have that $X_{t+1} \sim\break  \operatorname{Bin}(n, b(x))$.
So
\begin{eqnarray*}
\mathbb{E}\bigl[q_1^{X_{t+1}}|X_t=x\bigr] &=&
\bigl(b(x)q_1+1-b(x)\bigr)^n=\bigl(1-b(x)
(1-q_1)\bigr)^n
\\
&\leq  & e^{-nb(x)(1-q_1)} 
= \bigl[e^{-\lambda(1-q_1) e^{-{\lambda x}/{n}}}\bigr]^x
\leq q_1^x,
\end{eqnarray*}
where the last inequality follows by the definition of $q_1$.
This implies that $( q_1^{X_t } )_{t=0}^{T_{\varepsilon n}^+}$ is a
supermartingale.
We may apply the optional stopping theorem,
\[
q_1^x \geq\mathbb{E}\bigl[q_1^{X_{T_0\wedge T_{\varepsilon n}^+}}
\bigr]= \mathbb{E}\bigl[q_1^{X_{T_0}} \mathbh{1}_{ \{T_0<T_{\varepsilon n}^+
 \} }
\bigr]+ \mathbb{E}\bigl[q_1^{X_{T_{\varepsilon n}^+}} \mathbh{1}_{ \{
T_0>T_{\varepsilon n}^+  \} }
\bigr] \geq g(x)+\bigl(1-g(x)\bigr)q_1^n,
\]
and therefore
$g(x) \leq\frac{q_1^x-q_1^n}{1-q_1^n}$.

We obtain the lower bound similarly:
\[
\mathbb{E}\bigl[q_2^{X_{t+1}}|X_t=x\bigr] =
\bigl(b(x)q_2+1-b(x)\bigr)^n=\bigl(1-b(x)
(1-q_2)\bigr)^n.
\]
Now, $b(x)(1-q_2)=\frac{\lambda x}{n}e^{-{\lambda x}/{n}}(1-q_2)\leq
\lambda\varepsilon e^{-{\lambda x}/{n}}(1-q_2)\leq\lambda
\varepsilon< \frac{1}2$, so a short calculation gives
\begin{eqnarray*}
\mathbb{E}\bigl[q_2^{X_{t+1}}|X_t=x\bigr] &\geq &
e^{-nb(x)(1-q_2)(1+2b(x)(1-q_2))} 
\\
&=& \bigl[e^{- \lambda e^{-{\lambda x}/{n}} (1-q_2)(1+2 b(x) (1-q_2) ) }
\bigr]^x
\\
&\geq & \bigl[e^{- \lambda(1-q_2)(1+2\lambda\varepsilon) } \bigr]^x =
q_2^x.
\end{eqnarray*}
This implies that $(q_2^{X_t} )_{t=0}^{ T_{\varepsilon n}^+}$ is a
submartingale.
As before, by the optional stopping theorem,
\begin{eqnarray*}
q_2^x &\leq & \mathbb{E}_x\bigl[q_x^{X_{T_0\wedge T_{\varepsilon n}^+}}
\bigr]= \mathbb{E}\bigl[q_2^{X_{T_0}}\mathbh{1}_{ \{T_0<T_{\varepsilon n}^+
 \} }
\bigr]+ \mathbb{E}\bigl[q_2^{X_{T_{\varepsilon n}^+}}\mathbh{1}_{ \{
T_0>T_{\varepsilon n}^+  \} }
\bigr]\\
&\leq &  g(x)+\bigl(1-g(x)\bigr)q_x^{\varepsilon n}
\end{eqnarray*}
and therefore
$\frac{q_2^x-q_2^{\varepsilon n}}{1-q_2^{\varepsilon n}}\leq g(x)$.
\end{pf}

\begin{pf*}{Proof of Theorem~\ref{thm:extinction time}}
Fix $\varepsilon= \varepsilon(\lambda)>0$ small enough so that:
\begin{itemize}
\item It meets the requirements of Lemma~\ref{lem:sub and super}.

\item It satisfies $b(\varepsilon n)\leq b(n)$, or equivalently,
$\varepsilon\leq e^{\lambda(1-\varepsilon)}$.
It follows that $b(\varepsilon n)\leq b(x)$ for all\vspace*{1.5pt} $x\geq\varepsilon n$.

\item It satisfies $\varepsilon n \sqrt{\lambda}\leq n b(\varepsilon
n)$, or equivalently, $e^{\lambda\varepsilon}\leq\sqrt{\lambda}$.
\end{itemize}

Keeping in mind that $\mathbb{P}_y[T_0<T^+_{\varepsilon n}]=0$ for any
$y\geq\varepsilon n$,
by the Markov property we have that
$ \mathbb{P}_x [ T_0 < \tau_{\varepsilon n}^+ ] \leq
\mathbb{P}_x [ X_1 < \varepsilon n ]$ for all $x$.

Next, we bound the term $\mathbb{P}_x[X_1<\varepsilon n]$ using
standard large deviations for the binomial
distribution. Note\vadjust{\goodbreak} that by our choice of $\varepsilon$, for any $x\geq
\varepsilon n$ we have that
$\mathbb{E}_x[X_1]=nb(x) \geq nb(\varepsilon n)\geq\varepsilon n
\sqrt{\lambda}$. Therefore, for any $x\geq\varepsilon n$,
%
\begin{equation}\label{eqn:large down step}
\hspace*{6pt}\mathbb{P}_x \bigl[ T_0 < \tau_{\varepsilon n}^+ \bigr]
\leq\mathbb {P}_x[X_1<\varepsilon n]  \leq
\mathbb{P}_x \bigl[X_1< \lambda^{-1/2} \cdot
\mathbb {E}_x[X_1] \bigr] 
\leq\exp(-c\varepsilon n),
\end{equation}
where $c=\sqrt{\lambda} \cdot\frac{(1-\lambda^{-1/2})^2}{4}$.

Note that by Lemma~\ref{lem:sub and super} we have that for all $x>0$,
$\mathbb{P}_x [ T_{\varepsilon n}^+ < T_0 ] \geq c':= 1 - q_1
(1-q_1)^{-1} > 0$ where
$q_1 = q(\lambda e^{-\lambda\varepsilon} )$.
Thus, for $x < \varepsilon n$ we have that
\[
\mathbb{E}_x [ T_0 ] \geq c' \cdot\inf
_{y \geq\varepsilon n} \mathbb{E}_y [ T_0 ].
\]
So it remains to consider $x \geq\varepsilon n$.

By \eqref{eqn:large down step} and the strong Markov property,
on the event $X_0 \geq\varepsilon n$, the random time~$T_0$ dominates
a geometric random variable with success probability $e^{-c \varepsilon n}$.
Thus, for all $x$,
\[
\mathbb{E}_x [ T_0 ] \geq c' \cdot
e^{c \varepsilon n},
\]
which proves the theorem.
\end{pf*}

\section{Bounds on hitting probabilities}
\subsection{Probability of extinction before going above level
\texorpdfstring{$\varepsilon n$}{varepsilon n}}
Throughout this subsection we denote $g(x):=\mathbb
{P}_x[T_0<T_{\varepsilon n}^+]$.

Let $\alpha=\alpha(\lambda)\in(0,1)$ such that the following
inequalities hold:
$(1-\alpha)\lambda>1, \lambda e^{-\alpha\lambda}<1$.
This is equivalent to
$\frac{\log\lambda}{\lambda} <\alpha< 1-\lambda^{-1}$
which is possible since $\lambda> 1$.

Next, let $p(x,y)$ be the transition function of our Markov chain.
Explicitly, for any $0\leq x,y\leq n$, $p(x,y)=\mathbb{P}[\operatorname{Bin}(n,\break b(x))=y]$. Let $m(x):=\mathbb{E}[\operatorname{Bin}(n,b(x))]=nb(x)$
and $m_0(x):=(1-\alpha)m(x)$.

\begin{lem}
\label{lem:gamma}
For any $0<\varepsilon< \frac{1}{\lambda}$, $0\leq x<\varepsilon
n-1$ and $0\leq y\leq m_0(x)$, we have that
\[
p(x+1,y)\leq\gamma\cdot p(x,y),
\]
where $\gamma= \gamma_{\varepsilon,\alpha}:=e^{-\alpha\lambda
e^{-\lambda\varepsilon}(1-\lambda\varepsilon)} < 1$.
\end{lem}

\begin{pf}
Recall that $b(x)=\frac{\lambda x}{n}e^{-{\lambda x}/{n}}$. The
function $te^{-t}$ is increasing for $0\leq t<1$, which implies that
$b(x)$ is increasing while $\frac{\lambda x}{n}<1$,
and thus increasing as long as $x < \varepsilon n$.
It now follows that $\frac{p(x+1,y)}{p(x,y)}$ is increasing in $y$,
\[
\frac{p(x+1,y)}{p(x,y)} 
= \biggl( \frac{b(x+1)}{b(x)}
\biggr)^y \biggl(\frac{1-b(x+1)}{1-b(x)} \biggr)^{n-y},
\]
and since $b(x+1)>b(x)$, this expression is indeed increasing in $y$.
It follows that
\[
\max_{0\leq y\leq m_0(x)}\frac{p(x+1,y)}{p(x,y)} =\frac{p(x+1,m_0(x))}{p(x,m_0(x))}.
\]
So we want to bound from above the expression
$\frac{p(x+1,m_0(x))}{p(x,m_0(x))}$ with a bound that is independent
of $x$.

It can be simply checked that for
\[
f(t) = f_x(t) :=m_0(x)\log b(t) +\bigl(n-m_0(x)
\bigr)\log\bigl(1-b(t)\bigr),
\]
we have
\[
\log\frac{ p(x+1, m_0(x) ) }{ p(x,m_0(x) ) } = f(x+1) - f(x).
\]
So we want to bound $f(t+1)-f(t)$.
By the mean value theorem, it will be sufficient to bound $f'(t)$.

Recall that $x+1<\varepsilon n$, and $\varepsilon< \lambda^{-1}$, so
$b(\cdot)$ is monotone increasing
for $t \leq x+1$.
Upon differentiation, we get for all $t \in[x,x+1]$,
\begin{eqnarray*}
f'(t) &= & b'(t) \biggl(\frac{m_0(x)}{b(t)}-
\frac{n-m_0(x)}{1-b(t)} \biggr) \leq b'(t) \biggl(\frac{m_0(x)}{b(x) } -
\frac{n-m_0(x) }{ 1- b(x) } \biggr)
\\
& =&  \lambda \biggl(1- \frac{\lambda t}{n} \biggr) e^{- {\lambda t}/{n} } \cdot (-
\alpha) \cdot \biggl(1 + \frac{ m(x)}{n - m(x) } \biggr) \leq- \alpha\lambda(1 -
\lambda\varepsilon) e^{- \lambda
\varepsilon}.
\end{eqnarray*}
Thus
\begin{eqnarray*}
\max_{0\leq y\leq m_0(x)}\frac{p(x+1,y)}{p(x,y)} &=& \frac{p(x+1,m_0(x))}{p(x,m_0(x))}
= e^{f(x+1)-f(x)}
\\
&\leq &  e^{-\alpha\lambda e^{-\lambda\varepsilon}(1-\lambda\varepsilon
)} = \gamma_{\varepsilon, \alpha}.
\end{eqnarray*}
\upqed\end{pf}

\begin{lem}
\label{lem:beta}
There exist constants $\eta=\eta(\lambda)>0$ and
$0<\beta=\beta(\lambda)<\frac{1}{\lambda}$ such that for any
$0 < \varepsilon\leq\eta$ there exists $n_0=n_0(\varepsilon)$ such
that for all $n>n_0$,
we have that
\[
g(x+1)\leq\beta g(x) \qquad \forall x\geq0.
\]
\end{lem}

\begin{pf}
Recall that $g(x)=\mathbb{P}_x[T_0<T_{\varepsilon n}^+]$. It follows
immediately that $g(x)=0$ for $x\geq\varepsilon n$. Therefore, we only
consider $0\leq x<\varepsilon n$.

We have by the Markov property,
for $x+1 < \varepsilon n$,
%
\begin{eqnarray}
g(x+1) &=&  \sum_y p(x+1,y)g(y)
\nonumber
\\[-8pt]
\label{eqn:decompose h}
\\[-8pt]
\nonumber
& \leq &  \sum_{y \leq m_0(x)} p(x+1,y)g(y)+\sum
_{m_0(x) < y < \varepsilon n} p(x+1,y) g(y).
\end{eqnarray}

We bound the first term in \eqref{eqn:decompose h} using Lemma~\ref{lem:gamma}:
\[
\sum_{y \leq m_0(x)} p(x+1,y)g(y)  \leq\gamma_{\varepsilon,\alpha
}
\cdot \sum_{y \leq m_0(x)} p(x,y) g(y) \leq
\gamma_{\varepsilon,\alpha} \cdot g(x).
\]

For the second term, we use the upper bound of Lemma~\ref{lem:sub and super}
to obtain
\[
\sum_{m_0(x) < y < \varepsilon n }  p(x+1,y) g(y) \leq \sum
_{m_0(x) < y < \varepsilon n } p(x+1,y) \cdot\frac{q_1^y }{ 1 - q_1^n } \leq\frac{q_1^{m_0(x)}}{ 1
- q_1^n}.
\]
Also by Lemma~\ref{lem:sub and super},
\begin{eqnarray*}
\frac{q_1^{m_0(x)}}{ 1 - q_1^n } \cdot\frac{1}{g(x)} & \leq & \frac{1 - q_2^{\varepsilon n} }{(1 - q_1^n) (1 - q_2^{\varepsilon n -
x }) } \cdot
\frac{ q_1^{m_0(x)} }{ q_2^x }
\\
& \leq & \frac{1}{(1-q_1) (1-q_2) } \cdot \bigl(q_1^{(1-\alpha) \lambda e^{-\lambda\varepsilon} } /
q_2 \bigr)^x.
\end{eqnarray*}

Note that as $\varepsilon\to0$ we have that $q_1 \to q(\lambda
),q_2\to q(\lambda)$ and $e^{-\lambda\varepsilon}\to1$. Combined
with the assumption that $(1-\alpha)\lambda>1$, we can deduce that
there exists $\eta'>0$ such that
$q_1^{(1-\alpha)\lambda e^{-\lambda\varepsilon}} / q_2$ is bounded
away from $1$ uniformly in $0<\varepsilon\leq\eta'$.
Moreover, since $\gamma_{\varepsilon,\alpha} = e^{- \alpha\lambda
e^{- \lambda\varepsilon} (1-\lambda\varepsilon)}$,
and since we assume that $\lambda e^{- \alpha\lambda} <1$,
we may take $\eta'$ small enough so that for all $0 < \varepsilon\leq
\eta'$ we have
$\lambda\gamma_{\varepsilon,\alpha} < 1$.
Consequently,
we can find $K$ large enough (that depends only on $\eta'$)
such that 
\[
\beta' : = \sup_{\varepsilon\leq\eta'} \biggl(
\gamma_{\varepsilon
,\alpha} + \frac{1}{(1-q_1) (1-q_2) } \cdot \biggl(\frac
{q_1^{(1-\alpha)\lambda e^{-\lambda\varepsilon}}}{q_2}
\biggr)^K \biggr) < \frac{1}\lambda.
\]

Plugging all this into \eqref{eqn:decompose h},
we conclude that there exist $\eta'$ and $K\geq1$ and $\beta' <
\lambda^{-1}$
such that for all $0<\varepsilon\leq\eta'$ and for every $K\leq x <
\varepsilon n-1$,
we have $h(x+1)\leq\beta' h(x)$. This proves the lemma for $x \geq K$.

As for $0\leq x<K$, by Lemma~\ref{lem:sub and super} we have
%
\[
\frac{g(x+1)}{g(x)} \leq  \frac{q_1^{x+1}-q_1^n}{1-q_1^n}\cdot\frac{1-q_2^{\varepsilon n}}{
q_2^x-q_2^{\varepsilon n}}
 \leq    q_1 \cdot \biggl( \frac{q_1}{q_2}
\biggr)^x \cdot\frac{1}{
(1-q_2^{\varepsilon n-K})(1-q_1^n)}.
\]
%
Recall that $q_1>q_2$ so
$(q_1/q_2)^x \leq(q_1/q_2)^K \to1$ as $\varepsilon\to0$.
Also,
${1}/((1- q_2^{\varepsilon n-K})(1-q_1^n)) \to1$ as $n\to\infty$
and
$\lambda q_1\to\lambda q(\lambda)<1$ as $\varepsilon\to0$.

Therefore,
we may choose $\eta''$ such that for all $0 < \varepsilon\leq\eta''$,
$\lambda q_1 \cdot( q_1 / q_2)^K < \frac{\lambda q(\lambda) + 1}{2}$.
Thus there exists $n_0 = n_0(\varepsilon)$ such that if $n \geq n_0$,
we have
\[
\lambda q_1 \cdot \biggl(\frac{q_1}{q_2} \biggr)^K
\cdot\frac{1}{
(1-q_2^{\varepsilon n-K})(1-q_1^n)}<1.
\]
So we can take
\[
\beta''= \sup_{\varepsilon\leq\eta''}
q_1 \cdot \biggl(\frac{q_1}{q_2} \biggr)^{K} \cdot
\frac{1}{(1-q_2^{\varepsilon n_0-K})(1-q_1^{n_0})}
\]
to obtain that $\lambda\beta'' < 1$, and
for all $0<\varepsilon\leq\eta''$, sufficiently large $n$ and $0\leq  x <K$,
we have $g(x+1)\leq\beta'' g(x)$.
Wrap up by setting $\eta=\min\{\eta',\eta''\}$ and $\beta=\max\{
\beta',\beta''\}$.
\end{pf}

\subsection{Probability of extinction before going above level \texorpdfstring{$u\gg\varepsilon n$}{u>>varepsilon n}}

\begin{lem}[(Uniform lower bound)]
\label{lem:uniform lower bound}
Fix $\lambda>1$. There exists $\kappa=\kappa(\lambda) > 0$ such
that for all $0 < u < n$ and $0 < x+1 < u$,
\[
\mathbb{P}_{x+1} \bigl[ T_0 < T_u^+ \bigr]
\geq\kappa \mathbb{P}_x \bigl[ T_0 < T_u^+
\bigr].
\]
\end{lem}

\begin{pf}
Let $X= (X_k)_{k\geq0}$, $Y=(Y_k)_{k\geq0}$ be two Markov chains
starting from $X_0=x,Y_0=x+1$, respectively, and with Markov transition
kernel $p(\cdot,\cdot)$.
Consider the following coupling:
\begin{itemize}
\item if $b(x)\leq b(x+1)$, then let $X_1\sim\operatorname{Bin}(n,b(x))$,
and given $X_1$,
\[
Y_1 = X_1+\operatorname{Bin} \biggl(n-X_1,
\frac{b(x+1)-b(x)}{1-b(x)} \biggr);
\]
\item if $b(x)> b(x+1)$, then let $Y_1\sim\operatorname{Bin}(n,b(x+1))$,
and given $Y_1$,
\[
X_1 = Y_1+\operatorname{Bin} \biggl(n-Y_1,
\frac{b(x)-b(x+1)}{1-b(x+1)} \biggr).
\]
\end{itemize}
Next, given $X_k,Y_k$ for $k\geq1$, if $X_k = Y_k$, then couple
$X_{k+1} = Y_{k+1}$,
and otherwise let $X_{k+1}, Y_{k+1}$ evolve independently.
Note that $X_1=Y_1$ implies $X_k=Y_k$ for all $k\in\mathbb{N}$.

By the mean value theorem,
\[
\bigl|b(x+1)-b(x)\bigr| \leq \sup_{y\in[x,x+1]}\bigl|b'(y)\bigr| =\sup
_{y\in[x,x+1]} \biggl| \frac{\lambda}{n}e^{-{\lambda y}/{n}}\biggl(1-
\frac{\lambda y}{n}\biggr) \biggr| \leq\frac{\lambda}{n},
\]
and since $b(z) \leq e^{-1}$, we get that $\frac{b(x+1) - b(x)
}{1-b(x)}, \frac{b(x) - b(x+1)}{1-b(x+1) }
\leq\frac{e \lambda}{(e-1) n}$.
We have
\begin{eqnarray*}
\frac{\mathbb{P}[Y\in \{T_0<T_{u}^+ \} ]}{\mathbb
{P}[X\in \{T_0<T_{u}^+  \}]} &\geq & \frac{\mathbb{P}[Y_1 = X_1,  X \in \{T_0<T_{u}^+
 \} ]}{\mathbb{P}[X\in \{T_0<T_{u}^+  \}]}
\\
&=&  \mathbb{P}\bigl[Y_1=X_1 | X\in \bigl
\{T_0<T_{u}^+ \bigr\} \bigr].
\end{eqnarray*}

Now, if $b(x) \leq b(x+1)$, then as $n \to\infty$,
\begin{eqnarray*}
&& \mathbb{P}\bigl[ Y_1 = X_1 |  X \in \bigl\{T_0 < T_u^+ \bigr\} \bigr]
\\
&&\qquad \geq \sum_k \mathbb{P}\bigl[
X_1 = k | X \in \bigl\{T_0 < T_u^+
\bigr\} \bigr] \cdot \mathbb{P}\biggl[ \operatorname{Bin} \biggl( n - k,
\frac{b(x+1) - b(x) }{1-b(x)} \biggr) =0\biggr]
\\
&&\qquad \geq\mathbb{P}\biggl[ \operatorname{Bin} \biggl( n, \frac{e \lambda}{(e-1) n}
\biggr) =0 \biggr] \to e^{- e \lambda/ (e-1) } > 0.
\end{eqnarray*}
Similarly when $b(x) > b(x+1)$,
\[
\mathbb{P}\bigl[ Y_1 = X_1 | X \in \bigl\{T_0 < T_u^+ \bigr\} \bigr]  \geq \mathbb{P}\biggl[
\operatorname{Bin} \biggl( n, \frac{e \lambda}{(e-1) n} \biggr) =0 \biggr].
\]
We may take $\kappa: = \inf_n \mathbb{P}[ \operatorname{Bin} ( n, \frac
{e \lambda}{(e-1) n} ) =0 ]$
to complete the proof.
\end{pf}

\begin{lem}[(Geometric upper bound)]
\label{lem:geometric upper bound}
Fix $\lambda>1$ and $\varepsilon>0$ small. Then
there exists $\theta=\theta(\lambda,\varepsilon)\in(0,1)$ such
that for all $0 \leq x < u \le\mathrm{eq}- \varepsilon n$,
\[
\mathbb{P}_x \bigl[ T_0 < T_{u}^+ \bigr]
\leq\theta^x.
\]
\end{lem}

\begin{pf}
Consider the probability generating function of a
$\operatorname{Poi}(e^{\lambda\varepsilon})$ random variable.
Since $e^{\lambda\varepsilon}>1$, this function has a unique
nontrivial fixed point
$0 < \theta< 1$, satisfying $\theta=e^{-e^{\lambda\varepsilon
}(1-\theta)}$.
[Here $\theta= q(e^{\lambda\varepsilon})$ is the probability of
extinction of a Galton--Watson process
with Poisson-$e^{\lambda\varepsilon}$ offspring distribution.]

Note that for any $0\leq x<u \le\mathrm{eq}- \varepsilon n$, we have
$\lambda e^{-\lambda x/n} \geq e^{\lambda\varepsilon}$, so
\begin{eqnarray*}
\mathbb{E}\bigl[\theta^{X_{k+1}} | X_k=x\bigr]&=& \bigl(1-b(x)
(1-\theta)\bigr)^n \leq e^{-nb(x)(1-\theta)} 
\\
& =&  \bigl[ e^{-\lambda e^{-{\lambda x}/{n}}(1-\theta)} \bigr]^x \leq \bigl[ e^{-e^{\lambda\varepsilon}(1-\theta)}
\bigr]^x =\theta^x.
\end{eqnarray*}
%
This implies $(\theta^{X_k} )_{k=0}^{ T_{u}^+}$ is a supermartingale.
Since it is bounded, and $T_0\wedge T_{u}^+$ is a.s. finite, we may
apply the optional stopping theorem to this supermartingale with
$X_0=x$ for some $0\leq x<u$. We obtain
\begin{eqnarray*}
\theta^x  &\geq & \mathbb{E}_x\bigl[\theta^{X_{T_0\wedge T_{u}^+}}
\bigr] =\mathbb{E}_x\bigl[\theta^{X_{T_0}}\mathbh{1}_{ \{T_0<T_{u}^+
 \} }
\bigr] +\mathbb{E}_x\bigl[\theta^{X_{T_{u}^+}}\mathbh{1}_{ \{T_0>T_{u}^+
 \} }
\bigr] \\
&\geq &  \mathbb{P}_x \bigl[ T_0 <
T_{u}^+ \bigr].
\end{eqnarray*}
\upqed\end{pf}

\subsection{Coupling with subcritical branching process}

We want to investigate what our process behaves like when conditioned
on $T_0<T_{u}^+$ for $u = \varepsilon n$ and $u = \mathrm{eq}-\varepsilon n$. Let $\varphi(x) = \varphi_u(x) := \mathbb
{P}_x[T_0 < T_u^+]$. We denote the transition matrix of the tilted
chain as $p_\varphi(\cdot,\cdot)$ which is obtained by applying
Doob's $h$-transform to the original transition matrix $p(\cdot,\cdot)$
w.r.t. the harmonic function $\varphi$.
The matrix $p_\varphi$ is given by
%
\begin{equation}
\label{eq:doob}
\hspace*{3pt}\quad p_\varphi(x,y)=\mathbb{P}_x
\bigl[X_1=y|T_0<T_{u}^+\bigr]=
\frac{\mathbb{P}_x[X_1=y,T_0<T_{u}^+]}{\mathbb{P}_x[T_0<T_{u}^+]} =\frac{\varphi(y)p(x,y)}{\varphi(x)}.
\end{equation}

\begin{lem}
\label{lem:st1}
Let $0 < u < n$ and $\varphi(x) = \mathbb{P}_x[T_0 < T_u^+]$.
Suppose $\varphi(y+1)\leq\beta\varphi(y)$ for some $\beta>0$ and
for all $y\geq0$. Then for any $0\leq x< u$, the probability measure
$p_\varphi(x,\cdot)$ is stochastically dominated by the probability
measure $\mu_x$, where
\[
\mu_x(y)\propto\beta^y p(x,y),\qquad  y\geq0.
\]
\end{lem}

\begin{pf}
Let $Y \sim p_\varphi(x,\cdot), Z \sim\mu_x$.
We need to show that for any $0\leq x< u$ and $k\geq0$,
$\mathbb{P}[Y\leq k]\geq\mathbb{P}[Z\leq k]$,
or equivalently,
\[
\frac{\sum_{y=0}^{k}p(x,y)\varphi(y)}{\varphi(x)}- \frac{\sum_{y=0}^{k}p(x,y)\beta^y}{\sum_{z=0}^{\infty} p(x,z)\beta^z} \geq0.
\]
Since $\varphi(x)$ is harmonic with respect to $p(\cdot,\cdot)$,
this is equivalent to showing\vspace*{-2pt} that
\[
\frac{\sum_{y=0}^{k}p(x,y)\varphi(y)}{\sum_{z=0}^{\infty}
p(x,z)\varphi(z)}- \frac{\sum_{y=0}^{k}p(x,y)\beta^y}{\sum_{z=0}^{\infty} p(x,z)\beta^z} \geq0.
\]
So\vspace*{-2pt} we write
\begin{eqnarray*}
&& \sum_{y=0}^{k}  p(x,y)\varphi(y)\sum
_{z=0}^{\infty} p(x,z)\beta ^z-\sum
_{y=0}^{k}p(x,y)\beta^y\sum
_{z=0}^{\infty} p(x,z)\varphi (z)
\\[-3pt]
&& \qquad =
\sum_{y=0}^{k}\sum
_{z=k+1}^{\infty} p(x,y)p(x,z) \bigl(\varphi(y)
\beta^z-\varphi(z)\beta^y\bigr).
\end{eqnarray*}
%
Note that for $0\leq y<z$, by the assumption,
\[
\varphi(z)\beta^y\leq\varphi(y)\beta^{z-y}
\beta^y=\varphi (y)\beta^z,
\]
which implies that each term of the above sum is nonnegative.
\end{pf}

\begin{lem}
\label{lem:st3}
Let $0 < u < n$ and $\varphi(x) = \mathbb{P}_x[T_0 < T_u^+]$. Suppose
$\varphi(y)\geq\kappa\varphi(y-1)$ for some $\kappa>0$ and for all
$0 < y < u$. Then for any $0\leq x< u$, the probability measure
$p_\varphi(x,\cdot)$ stochastically dominates the probability measure
$\nu_x$, where\vspace*{-4pt}
\[
\nu_x(y)\propto\kappa^y p(x,y) \mathbh{1}_{ \{y < u \}}.
\]
\end{lem}
\begin{pf}
The proof is exactly similar to that of Lemma~\ref{lem:st1} where we
replace ``$\infty$'' in the bounds of the summands by $u-1$. We omit
the details.
\end{pf}

\begin{lem}
\label{lem:st2}
\textup{(a)} Fix $0<p<1$. Let $X\sim\operatorname{Bin}(n,p)$ and $Y\sim\operatorname{Poi}(-n\log(1-p))$.
Then $X\leq_{\mathrm{st}}Y$.

\textup{(b)} Fix $0< p_1 < p_2 < 1$. Let $X\sim\operatorname{Bin}(n,p_1)$ and $X\sim
\operatorname{Bin}(n,p_2)$. Then for any $m \ge0$, we have $X | \{X \le m\}
\leq_{\mathrm{st}}Y | \{Y \le m\}$.
\end{lem}

\begin{pf}
(a) We exhibit a coupling such that $X\leq Y$. Note that
$X=X_1+\cdots+X_n$ where $X_1,\ldots,X_n$ are i.i.d.
$\operatorname{Ber}(p)$ random variables, and $Y=Y_1+\cdots+Y_n$ where
$Y_1,\ldots,Y_n$ are i.i.d.
$\operatorname{Poi}(-\log(1-p))$ random variables. So let $Y_1,\ldots,Y_n$ be
as such, and let $X_j : = \mathbh{1}_{ \{Y_j > 0  \} }$.
It follows that $X_j \leq Y_j$ and
$\mathbb{P}[ X_j = 0 ] = \mathbb{P}[ Y_j = 0 ] = (1-p)$, suggesting
that indeed $X_1,\ldots, X_n$ are i.i.d. $\operatorname{Ber}(P)$, and $X \leq Y$ a.s.

(b) It suffices to show for any $k \le m$,
\[
\frac{\mathbb{P}[ X \le k]}{\mathbb{P}[ X \le m]} \ge\frac{\mathbb
{P}[ Y \le k]}{\mathbb{P}[ Y \le m]},
\]
which, in turn, is implied by $\mathbb{P}[ X = i] \mathbb{P}[ Y = j]
\ge\mathbb{P}[ X = j] \mathbb{P}[ Y = i] $ for all $0\le i < j \le
m$. Upon rearrangement of terms, the above is equivalent to
\[
\frac{ ( {p_1}/({1-p_1})  )^i}{  ( {p_2}/({1-p_2})  )^i} \ge \frac{ ( {p_1}/({1-p_1})
)^j}{  ( {p_2}/({1-p_2}))^j},
\]
which is obviously true.
\end{pf}

\begin{lem}
\label{lem:coupling1}
Denote the Markov chain conditioned on $T_0<T_{\varepsilon n}^+$ by\break 
$(X'_m)_{m \ge0}$.
There exists $\varepsilon_0>0$ and $0<\bar{\gamma}< 1$
such that for all $\varepsilon< \varepsilon_0$ there exists $n_0 =
n_0(\varepsilon)$
such that for all $n > n_0$ the following holds.
For any $0 \le x_0 < \varepsilon n$, we can couple $ ( X'_m )_{m\geq0}$
with a sub-critical Galton--Watson process $( W_m )_{m\geq0}$ having
offspring distribution
$\operatorname{Poi}(\bar{\gamma})$, such that
\[
X'_0 = W_0 = x_0,\qquad
X'_m\leq W_m\qquad \forall m\geq1.
\]
\end{lem}

\begin{pf}
Recall that the transition matrix of the conditioned chain $X'$ is
given by $p_\varphi(\cdot, \cdot)$ with $\varphi(x) = \mathbb
{P}_x[T_0 < T_{\varepsilon n}^+]$.
It suffices to show that for any nonnegative integers $x \leq w$,
$p_\varphi(x,\cdot)$ is stochastically dominated by a $\operatorname{Poi}(w \gamma)$ random variable.

By Lemmas \ref{lem:beta} and \ref{lem:st1}, we know that for small
enough $\varepsilon>0$ and large enough $n$, $p_\varphi(x,\cdot)$ is
stochastically dominated by $\mu_x(\cdot)$. Fix $0\leq x<\varepsilon
n$. Note that
\[
\mu_x(y) \propto   \beta^y p(x,y) \propto \pmatrix{n
\cr
y}
\biggl( \frac
{\beta b(x)}{1-b(x)} \biggr)^y,\qquad  0\leq y\leq n,
\]
which implies that $\mu_x(\cdot)$ is binomially distributed with $n$
trials and success probability $\theta(x)$ that satisfies
\[
\frac{\theta(x)}{1-\theta(x)}=\frac{\beta b(x)}{1-b(x)} \quad \mbox{or}\quad \theta(x)=\frac{\beta b(x)}{1-b(x)(1-\beta)}.
\]
Further, by Lemma~\ref{lem:st2}(a), $\mu_x(\cdot)$ is stochastically
dominated by a Poisson random variable with mean $g(x)=-n \log
(1-\theta(x))$.
Note that $b(x) \leq e^{-1} < \frac{1}2$. So $\beta b(x) < 1 - b(x)$,
and thus $\theta< \frac{1}2$.
Also, one easily checks that $-\log(1-t)\leq t+2 t^2$ for $t\in
[0,\frac{1}{2}]$.
We thus obtain
$g(x) \leq n \theta(x) (1 + 2 \theta(x) )$.
Recall that $b(x)=\frac{\lambda x}{n} e^{-{\lambda x}/{n}}$ is monotone
on $[0,\varepsilon n]$ so for $x < \varepsilon n$, we have $b(x)\leq
\lambda\varepsilon$.
Hence
\[
g(x)\leq x \cdot\frac{\beta\lambda}{1-\lambda\varepsilon} \cdot \biggl(1+\frac{2 \lambda\varepsilon}{1-\lambda\varepsilon} \biggr).
\]
Denote
$\bar{\gamma}:=\frac{\beta\lambda}{1-\lambda\varepsilon}(1+\frac
{\lambda\varepsilon}{1-\lambda\varepsilon})$.
Note that for $\varepsilon\to0$ we have that $\bar{\gamma}\to
\lambda\beta<1$ by
Lemma~\ref{lem:beta}. So choose $\varepsilon_0$ small enough so that
$\bar{\gamma}<1$. Putting all the ingredients together, we get that
for all $0\leq x<\varepsilon n$,
\[
p_\varphi(x,\cdot)\leq_{\mathrm{st}} \operatorname{Poi}(x\bar{\gamma}),
\]
and keeping in mind that for any two nonnegative integers $x\leq w$ we have
$\operatorname{Poi}(x\bar{\gamma})\leq_{\mathrm{st}} \operatorname{Poi}(w \bar{\gamma})$,
the proof is complete.
\end{pf}

\begin{cor}
\label{cor:extinction from small}
Fix $\lambda>1$. There exist $\varepsilon_0>0$
such that for all $\varepsilon< \varepsilon_0$ there exists $C>0$
such that the following holds for all $n \ge1$:
\[
\mathbb{E}_x\bigl[T_0|T_0<T_{\varepsilon n}^+
\bigr] \leq C \log(1+x),\qquad  0\leq x<\varepsilon n.
\]
\end{cor}

\begin{lem}
\label{lem:coupling2}
Denote the Markov chain conditioned on $T_0<T_{\mathrm{eq}-
\varepsilon n}^+$ by $(X''_m)_{m \ge0}$. Given $\lambda>1$ and
$0<\varepsilon< \frac{\log\lambda}{\lambda}$, there\vspace*{1pt} exists
$0<\underline{\gamma}< 1$
such that the following holds.
For any $0 \le x_0 < \mathrm{eq}- \varepsilon n$, we can couple $ (
X''_m )_{m\geq0}$
with a subcritical Galton--Watson process $ ( V_m )_{m\geq0}$ having
offspring distribution
$\operatorname{Ber}(\underline{\gamma})$, such that with probability at
least $1 - e^{-(1- \kappa)^2 (\lambda^{-1} \log\lambda- \varepsilon)^2n}$,
\[
X''_0 = V_0 =
x_0,\qquad  X''_m \geq
V_m \qquad \forall 1 \le m\leq e^{(1- \kappa)^2 (\lambda^{-1} \log\lambda- \varepsilon)^2n},
\]
where $\kappa\in(0,1)$ is as given in Lemma~\ref{lem:uniform lower
bound} with $u = \mathrm{eq}- \varepsilon n$.
\end{lem}

\begin{pf}
By Lemma~\ref{lem:uniform lower bound} and Lemma~\ref{lem:st3}, for
any $0 < x < \mathrm{eq}- \varepsilon n$, the transition distribution
$p_\varphi(x, \cdot)$ stochastically dominates $\nu_x(y)\propto{n
\choose y}  (\frac{b(x)\kappa}{1 - b(x)} )^y\times\break  \mathbh{1}_{ \{y < \mathrm{eq}-\varepsilon n  \} }$. In other
words,
\[
Y | \{Y < \mathrm{eq}- \varepsilon n\} \le_{\mathrm{st}} p_\varphi(x,
\cdot),
\]
where $Y$ is distributed as $\operatorname{Bin}(n, \theta(x))$, where
$ \theta(x) = \frac{\kappa b(x)}{1 - b(x) (1- \kappa)}$. By
Lem\-ma~\ref{lem:st2}(b) and from the simple\vspace*{1pt} inequality $\theta(x) \ge
\frac{ \kappa x}{n}$, we further have $ Z | \{Z < \mathrm{eq}-
\varepsilon n\} \le_{\mathrm{st}} p_\varphi(x, \cdot)$, where $Z$ is
distributed as $\operatorname{Bin}(n, \frac{ \kappa x}{n})$.
Clearly, $\sum_{i=1}^x Z_i \le_{\mathrm{st}} Z$ where $Z_i$ are i.i.d.
$\operatorname{Bin}( \lfloor\frac{n}{x} \rfloor, \frac{ \kappa x}{n})$.
We can find $\underline{\gamma}<1$ such that for any $n \ge1$ and
any $0< x < \mathrm{eq}- \varepsilon n$,
\[
\mathbb{P}[Z_i = 0] = \biggl(1 - \frac{ \kappa x}{n} \biggr)
^{ \lfloor {n}/{x} \rfloor} \le1- \underline{\gamma}.
\]
Consequently, $Z_i$ stochastically dominates $\operatorname{Ber}( \underline
{\gamma})$ and hence
$ \operatorname{Bin}(x, \underline{\gamma}) \le_{\mathrm{st}} Z$.

On the other hand, by Hoeffding's inequality,
\[
\mathbb{P}[ Z \ge\mathrm{eq}- \varepsilon n] \le\exp \bigl(-2(1-
\kappa)^2 \bigl(\lambda^{-1} \log\lambda- \varepsilon
\bigr)^2n \bigr).
\]
Thus, on the event $\{ Z < \mathrm{eq}- \varepsilon n\}$, which
happens with probability at least $1- \exp( - 2 (1- \kappa)^2(\lambda
^{-1} \log\lambda- \varepsilon)^2 n )$, the distribution $p_\varphi
(x, \cdot)$ stochastically dominates $ \operatorname{Bin}(x, \underline
{\gamma})$. So, a simple union bound allows us to couple the
conditioned chain $X''$ with a subcritical Galton--Watson process
having offspring distribution $\operatorname{Ber}(\underline{\gamma})$ so
that with probability at least $1 - \exp( -(1- \kappa)^2 (\lambda
^{-1} \log\lambda- \varepsilon)^2 n )$, the subcritical
Galton--Watson process is dominated by $X''$ for the first $\exp( (1-
\kappa)^2 (\lambda^{-1} \log\lambda- \varepsilon)^2 n )$ steps.
This completes the proof.
\end{pf}

\begin{cor}[(Lower bound on transition window)]
\label{cor:lb}
Given $\lambda>1$ and $0<\varepsilon< \frac{\log\lambda}{\lambda
}$, there exists $C>0$
such that the following holds for all $n \ge1$:
\[
\mathbb{E}_x\bigl[T_0|T_0<T_{\mathrm{eq}- \varepsilon n}^+
\bigr] \geq C^{-1} \log(1+x), \qquad 0\leq x< \mathrm{eq}- \varepsilon n.
\]
\end{cor}

\begin{pf}
By Lemma~\ref{lem:coupling2}, the conditioned chain $X''$ can be
coupled with the sub-critical Galton--Watson process $V$ with mean
offspring $\underline{\gamma}$ such that
$X''_0 = V_0 =x$ and $X''_t \ge V_t$ for all $1 \le t \le e^{cn}$ with
probability at least $1 - e^{-cn}$. Let $S$ be the time of extinction
for the process $V$. It follows from the standard theory of branching
process that $\mathbb{E}_x [S] \ge c' \log(1+x)$ for all $x \ge0$.
On the other hand,
\[
\mathbb{P}_x[S> t] \le x \mathbb{E}_1
[V_t ] \le x \underline{\gamma }^t.
\]
By choosing $t = D \log n$ with a sufficiently large constant $D =
D(\underline{\gamma})>0$, we obtain that $E_x [S \mathbh{1}_{
\{S \le D \log n  \} }] \ge\frac{c'}{2} \log(1+x)$ for all $0
\le x \le n$. By the above coupling,
\[
\mathbb{E}_x\bigl[T_0|T_0<T_{\mathrm{eq}- \varepsilon n}^+
\bigr] \ge\mathbb {E}_x [S \mathbh{1}_{ \{S \le D \log n  \} }] - D \log n
\cdot e^{-cn},
\]
which completes the proof.
\end{pf}

\section{Upper bound on extinction window}

Throughout this section we set
$\varepsilon>0$ and $h(x) = \mathbb{P}_x [ T_0 < T_{\mathrm{eq}-\varepsilon n}^+ ]$.

The next lemma bootstraps the result from Corollary~\ref{cor:extinction from small}.

\begin{lem}
\label{lem:starting from small}
Given $\lambda>1$ and $\varepsilon>0$, there exist $C>0$ and $ \eta>0$
such that for all $0 \leq x < \eta n$,
\[
\mathbb{E}_x \bigl[ T_0 |  T_0 <
T_{\mathrm{eq}-\varepsilon n}^+ \bigr] \leq C \log(1+x).
\]
\end{lem}

\begin{pf}
By Lemmas \ref{lem:uniform lower bound}
and \ref{lem:geometric upper bound},
there exist $\kappa<1$ and $\theta < 1$ such that for all $0 \le x <
\mathrm{eq}- \varepsilon n$,
$\kappa^x \leq h(x) \leq\theta^x$.
We choose $r>1$ so that $\theta^{r/2} < \frac{\kappa}{2}$.

By Corollary~\ref{cor:extinction from small}, there exist
$C>0$, $\eta' > 0$ small enough such that for all $x < \eta' n$,
\[
\mathbb{E}_x \bigl[ T_0  |  T_0 <
T_{\eta' n}^+ \bigr] \leq C \log(1+x).
\]

Then, taking $\eta= \eta' / r$,
by the strong Markov property, for any $0 \leq x < \eta n$,
\begin{eqnarray*}
\mathbb{P}_x \bigl[ T_{\eta' n}^+ < T_0 <
T_{\mathrm{eq}-\varepsilon n
}^+ \bigr] & \leq  & \sup_{x \geq\eta' n} h(x) \leq
\theta^{r \eta n} < \biggl(\frac
{\kappa}{2} \biggr)^{2 \eta n}
\\
& \leq & \inf_{x \leq\eta n} h(x)^2 \cdot2^{-2 \eta n}.
\end{eqnarray*}

Note that since for $x \leq\mathrm{eq}-\varepsilon n$,
\[
n b(x) = x \cdot\lambda e^{-\lambda x/n} \geq e^{\lambda\varepsilon} \cdot x,
\]
we have that if $X_k < u - \varepsilon n$, then $\mathbb{E}[ X_{k+1} |
X_k ] \geq e^{\lambda\varepsilon} \cdot X_k$.
Let $A = e^{ \lambda\varepsilon/ 2} > 1$,
and let $p = \exp (- \frac{(A-1)^2}{4}  ) < 1$.
By standard large deviations of binomial random variables,
\begin{eqnarray*}
\mathbb{P}_x [ X_1 \leq A x ] & \leq &
\mathbb{P}_x \bigl[ X_1 \leq A^{-1}
\mathbb{E}_x[X_1] \bigr] \leq\exp \biggl(-
\frac{(1 - A^{-1} )^2}{4} \cdot\mathbb{E}_x[ X_1] \biggr)
\\
& \leq & \exp \biggl(- \frac{(A-1)^2}{4} \cdot x \biggr) = p^x.
\end{eqnarray*}
Let $m>1$ be an integer such that $A^m \geq e$.
Then
\[
\mathbb{P}_x \bigl[ T_{ex}^+ \leq m \bigr]  \geq
\mathbb{P}_x \bigl[ \forall  1 \leq j \leq m \wedge
T_{ex}^+,  X_{j} \geq A X_{j-1} \bigr] \geq
\bigl(1 - p^x \bigr)^{m}.
\]
This holds for any $x$. Thus inductively, for $k = \lceil\log n \rceil$,
\begin{eqnarray*}
\mathbb{P}_x \bigl[ T_{\mathrm{eq}-\varepsilon n}^+ \leq k m \bigr] &\geq &
\mathbb{P}_x \bigl[ T_{ex}^+ \leq m \bigr] \cdot\inf
_{y \geq ex} \mathbb{P}_y \bigl[ T_{\mathrm{eq}-\varepsilon n}^+
\leq(k -1) m \bigr]
\\
&\ge & \cdots\geq(1-p)^{m \log n}.
\end{eqnarray*}
Again this holds for all $x>0$.
Thus $\frac{1}{km} \cdot(T_0 \wedge T_{\mathrm{eq}-\varepsilon n}^+)$
is dominated by a geometric random variable
of mean $(1-p)^{-m \log n}$.
So we conclude that
\[
\mathbb{E}_x \bigl[ \bigl(T_0 \wedge
T_{\mathrm{eq}-\varepsilon n}^+\bigr)^2 \bigr] \leq (km)^2 \cdot2
(1-p)^{-2 \log n},
\]
which is polynomial in $n$ as $n \to\infty$.

Let $\mathcal{E}(u)$ denote the event $ \{T_0 < T_u^+  \}$.
Putting everything together, we obtain that for $0<x <\eta n$,
\begin{eqnarray*}
&& \mathbb{E}_x \bigl[ T_0  |  \mathcal{E}(
\mathrm{eq}-\varepsilon n) \bigr]
\\
&&\qquad  = \mathbb{E}_x \bigl[ T_0 \mathbh{1}_{ \mathcal{E}(\eta' n) }
|  \mathcal{E}(\mathrm{eq}-\varepsilon n) \bigr] \\
&&\qquad\quad{}+ \mathbb{E}_x
\bigl[ T_0 \wedge T_{\mathrm{eq}-\varepsilon n}^+ \cdot \mathbh{1}_{ \{T_{\eta' n}^+ <T_0 < T_{\mathrm{eq}-\varepsilon
n}^+  \} }
 |  \mathcal{E}(\mathrm{eq}-\varepsilon n) \bigr]
\\
&&\qquad \leq\mathbb{E}_x \bigl[ T_0  |  \mathcal{E}\bigl(
\eta'n\bigr) \bigr] \cdot\frac
{\mathbb{P}_x [ \mathcal{E}(\eta' n) ] }{\mathbb{P}_x [ \mathcal
{E}(\mathrm{eq}-\varepsilon n) ] }\\
&&\qquad\quad{}+ \frac{ \sqrt{ \mathbb{E}_x [ (T_0 \wedge T_{\mathrm
{eq}-\varepsilon n}^+ )^2 ]
\cdot\mathbb{P}_x [ T_{\eta' n} < T_0 < T_{\mathrm{eq}-\varepsilon
n}^+ ] }}{ h(x) }
\\
&&\qquad  \leq C \log (1+x) + km\sqrt{2} (1-p)^{-\log n} \cdot2^{-\eta n} \leq
C' \log(1+x),
\end{eqnarray*}
for some constant $C'>0$.
\end{pf}

\begin{lem}
\label{lem:petunia}
Given $\lambda>1$ and $\varepsilon,\delta>0$, there exists a
constant $C=C(\varepsilon,\lambda,\break  \delta) > 0$
such that for all $0\leq x<\mathrm{eq}-\varepsilon n$,
\[
\mathbb{E}_x\bigl[H  |  T_0<T_{\mathrm{eq}-\varepsilon n}^+
\bigr]\leq C,
\]
where $H:=\sum_{k=0}^{\infty}\mathbh{1}_{ \{\delta
n<X_k< \mathrm{eq}-\varepsilon n  \} }$ is the total time spent
by $X$
in the interval $(\delta n, \mathrm{eq}-\varepsilon n)$.
\end{lem}

\begin{pf}
We start with the observation that for $0<x<\mathrm{eq}-\varepsilon n$,
since $\lambda e^{-\lambda\mathrm{eq}/ n} = 1$,
\[
\frac{nb(x)}{x}=\lambda e^{-{\lambda x}/{n}} \geq\lambda e^{-({\lambda(\mathrm{eq}-\varepsilon n)})/{n}}=
e^{\lambda\varepsilon}> 1. 
\]
Choose $m = m(\varepsilon, \lambda,\delta) \geq1$ large enough such that
$\delta e^{({\lambda\varepsilon}/{2})m}>\frac{\mathrm{eq}-
\varepsilon n}{n}$.
Call a step~$k$ of the Markov chain $X$ \textit{unusual}
if $\delta n<X_k<\mathrm{eq}-\varepsilon n$ and $X_{k+1}<e^{{\lambda\varepsilon}/{2} } X_k$.
By standard large deviations of binomial random variables, for $0<\xi<1$,
$\mathbb{P}[ \operatorname{Bin}(n,b) < \xi nb ] \leq\exp (- nb \cdot
\frac{(1-\xi)^2}{4}  )$.
Since $n b(X_k) \geq e^{\lambda\varepsilon} X_k$,
the probability of step $k$ being unusual is bounded by
\[
\mathbb{P}\bigl[ \operatorname{Bin}\bigl(n,b(X_k) \bigr) <
e^{-\lambda\varepsilon/2} n b(X_k)  |  \delta n < X_k <
\mathrm{eq}- \varepsilon n \bigr] \leq\exp \biggl(- \bigl(e^{\lambda\varepsilon/2}-1
\bigr)^2 \cdot\frac
{\delta n }{ 4} \biggr).
\]

Note that if $\delta n < X_j < u - \varepsilon n$ and for all $k=j,
j+1, \ldots, j+m-1$
the step~$k$ is \textit{not} unusual, then by our choice of $m$,
\[
X_{j+m} \geq e^{({\lambda\varepsilon}/{2}) m } X_j > \frac
{\mathrm{eq}-\varepsilon n }{\delta n}
\cdot\delta n = \mathrm{eq}- \varepsilon n.
\]
That is, if all steps $j,j+1,\ldots,j+m-1$ are not unusual, then $T_0
> T_{\mathrm{eq}-\varepsilon n}^+$.
Thus $T_0 < T_{\mathrm{eq}-\varepsilon n}^+$ implies\vspace*{1pt} that every time
$j$ that $\delta n < X_j < \mathrm{eq}-\varepsilon n$,
we must have that there exists $j \leq k \leq j+m-1$ such that $k$ is
an unusual step.
In conclusion, for any $0<x<\mathrm{eq}-\varepsilon n$,
\begin{eqnarray*}
\mathbb{P}_x\bigl[T_0<T_{\mathrm{eq}-\varepsilon n}^+, H> d\bigr]
& \leq & \mathbb{P}_x\bigl[\mbox{$X$ takes at least $\lfloor d / m
\rfloor$ unusual steps}\bigr]
\\
& \leq & \exp \biggl(- \bigl(e^{\lambda\varepsilon/2}-1\bigr)^2 \cdot
\frac
{\delta}{4} \cdot\biggl\lfloor\frac{d}{m} \biggr\rfloor\cdot n
\biggr).
\end{eqnarray*}
On the other hand, by Lemma~\ref{lem:uniform lower bound}, $h(x) \ge
\kappa^n$. Combining the above two observations, we obtain
that for any $0<x<\mathrm{eq}-\varepsilon n$,
\[
\mathbb{P}_x\bigl[ H \ge d  |  T_0<T_{\mathrm{eq}-\varepsilon n}^+
\bigr] \le e^{( K_1 - dK_2)n}
\]
for constants $K_1$ and $K_2$ which are functions of $\delta,m,
\varepsilon$ and $\lambda$.
Now the assertion of the lemma follows immediately from the representation
\[
\mathbb{E}_x\bigl[ H  |  T_0<T_{\mathrm{eq}-\varepsilon n}^+
\bigr] = \sum_{d=0}^\infty
\mathbb{P}_x\bigl[ H \ge d  |   T_0<T_{\mathrm{eq}-\varepsilon n}^+
\bigr].
\]
%
\upqed\end{pf}

\begin{pf*}{Proof of Theorem~\ref{thm:main thm}}
The lower bound is established in Corollary~\ref{cor:lb}.

Let us now prove the upper bound.
Let $\eta>0$ be as in Lemma~\ref{lem:starting from small}. For $x <
\eta n$, Theorem~\ref{thm:main thm} follows directly from Lemma~\ref
{lem:starting from small}.
For $x \ge\eta n$, by the strong Markov property,
%
\begin{equation}
\label{eq:smp_thm3}
\hspace*{6pt}\mathbb{E}_x \bigl[ T_0   |
T_0 < T_{\mathrm{eq}-\varepsilon n}^+ \bigr] \leq \mathbb{E}_x
\bigl[ T_{\eta n}^-  |  T_0 < T_{\mathrm{eq}-\varepsilon n} \bigr] +
\sup_{y < \eta n} \mathbb{E}_y \bigl[ T_0 |  T_0 < T_{\mathrm
{eq}-\varepsilon n}^+ \bigr].
\end{equation}
We can bound $\mathbb{E}_x [ T_{\eta n}^-  |  T_0 < T_{\mathrm
{eq}-\varepsilon n} ] \le1+ \mathbb{E}_x [ H  |  T_0 < T_{\mathrm
{eq}-\varepsilon n} ]$ where $H =\sum_{k=0}^{\infty
}\mathbh{1}_{ \{\eta n < X_k < \mathrm{eq}-\varepsilon n
\} }$,
and hence by Lemma~\ref{lem:petunia}, $\mathbb{E}_x [ T_{\eta n}^-
|  T_0 < T_{\mathrm{eq}-\varepsilon n} ] \le C_1$.

Therefore, from \eqref{eq:smp_thm3} and by Lemma~\ref{lem:starting
from small},
\[
\mathbb{E}_x \bigl[ T_0  |  T_0 <
T_{\mathrm{eq}-\varepsilon n}^+ \bigr] \leq C_1+ C \log( \eta n) \le
C' \log(1+x),
\]
which completes the proof of Theorem~\ref{thm:main thm}.
\end{pf*}

\section*{Acknowledgments}

This work was initiated when the second and third authors were in the
Statistical Laboratory, University of Cambridge. Thanks goes to
Nathana\"el Berestycki
for making this possible.  The authors also thank the anonymous referee
for helpful comments and suggestions.






\printaddresses
\end{document}